\let\FormatGMIG=G
\let\FormatNormal=N
\let\format=\FormatNormal

\documentclass[10pt]{article}

\if\format\FormatGMIG
	\usepackage{gmig}
\fi

\usepackage[margin=1in]{geometry}

\usepackage{std,amsthm,graphicx,alignenv,xcolor,soul,mathtools,verbatim,url,suffix,cite}
\usepackage{algorithm,algorithmic}
\usepackage[subrefformat=parens,labelformat=parens]{subfig}
\usepackage[utf8]{inputenc}
\usepackage{lmodern}
\usepackage[only,shortrightarrow]{stmaryrd}
\usepackage[bottom]{footmisc}

\usepackage{mathabx}
\usepackage[margin=10pt,font=small,labelfont=bf]{caption}

	\usepackage{charter}

\if\format\FormatNormal	
	\newcommand\CompilationAuthor[1]{}
	\newcommand\CompilationTitle[1]{}
\fi

\makeatletter
\def\blfootnote{\gdef\@thefnmark{}\@footnotetext}
\makeatother

\usepackage{scalerel,stackengine}
\stackMath
\newcommand\reallywidehat[1]{%
\savestack{\tmpbox}{\stretchto{%
  \scaleto{%
    \scalerel*[\widthof{\ensuremath{#1}}]{\kern-.6pt\bigwedge\kern-.6pt}%
    {\rule[-\textheight/2]{1ex}{\textheight}}
  }{\textheight}%
}{0.5ex}}%
\stackon[1pt]{#1}{\tmpbox}%
}

\usepackage{scalerel}
\newlength\bshft
\def\fakebold#1{\ThisStyle{\ooalign{$\SavedStyle#1$\cr%
  \kern-\bshft$\SavedStyle#1$\cr%
  \kern\bshft$\SavedStyle#1$}}}

\makeatletter
\newtheorem*{rep@theorem}{\rep@title}
\newcommand{\newreptheorem}[2]{%
\newenvironment{rep#1}[1]{%
 \def\rep@title{#2 \ref{##1}}%
 \begin{rep@theorem}}%
 {\end{rep@theorem}}}
\makeatother

\theoremstyle{plain}
\newtheorem{theorem}{Theorem}
\newtheorem{mainthm}{Theorem}
\newreptheorem{mainthm}{Theorem}

\newtheorem{proposition}[theorem]{Proposition}
\newtheorem{lemma}[theorem]{Lemma}
\newtheorem{corollary}[theorem]{Corollary}
		
\theoremstyle{definition}
\newtheorem*{definition}{Definition}

\theoremstyle{remark}

\newtheorem*{remark}{Remark}

\numberwithin{theorem}{section}
\numberwithin{example}{section}
\numberwithin{figure}{section}

\DeclareMathOperator\AC{AC}

\numberwithin{equation}{section}

\newcommand\nobelowdisplayskip
	{\setlength\belowdisplayskip{0pt}}						

\definecolor{yorange}{rgb}{1,0.8,0}
\sethlcolor{yorange}

\let\d\bdy 						
\DeclareMathOperator\WF{WF}		
\let\gradient\nabla				

\newcommand\En{\mathbf E}		
\newcommand\KE{\mathbf{KE}}		

\newcommand\DT{\ensuremath{_{\text{\normalfont DT}}}}		

\newcommand\RDT{\mathbf{DT}}				
\newcommand\dt{\mathbf{dt}}						

\newcommand\bgamma{\boldsymbol\gamma}

\newcommand\JCS{J_{\text C\shortrightarrow \text S}}		
\newcommand\JCB{J_{\text C\shortrightarrow \bdy}}		
\newcommand\JBS{J_{\bdy\shortrightarrow \text S}}		
\newcommand\JBB{J_{\bdy\shortrightarrow \bdy}}			

\newcommand\subT{_{\text T}}						

\newcommand\eqml{\Eq}					

\newcommand\Garding{G\aa{}rding}

\DeclareMathOperator\cone{cone}

\let\abs\anundefinedmacro
\let\norm\anundefinedmacro
\let\form\anundefinedmacro
\let\set\anundefinedmacro
\DeclarePairedDelimiter\abs\lvert\rvert
\DeclarePairedDelimiter\norm\lVert\rVert
\DeclarePairedDelimiter\form\langle\rangle
\DeclarePairedDelimiterX\set[2]{\{}{\}}{#1\,\delimsize\vert\,#2}

\newcommand\indicator{\mathbf 1}

\let\hat\widehat

\let\phi\varphi

\let\Re\RRe

\CompilationTitle{Reconstruction of piecewise smooth wave speeds using multiple scattering}
\title{\vspace{-1em}Reconstruction of piecewise smooth wave speeds using multiple scattering}

\CompilationAuthor{Peter Caday\and Maarten V.~de Hoop\and Vitaly Katsnelson\and Gunther Uhlmann}
\author{Peter Caday$^{\text{*}}$\!,\,
		Maarten V.~de Hoop$^{\text{\textdagger}}$\!,\,
		Vitaly Katsnelson$^{\text{\textdaggerdbl}}$\!,\, and
		Gunther Uhlmann$^\parallel$}
\date\today

\begin{document}

\if\format\FormatGMIG
		\GMIGTitle
\else
		\maketitle
\fi

\begin{abstract}
	Let $c$ be a piecewise smooth wave speed on $\RR^n$, unknown inside a domain $\Omega$. We are given the solution operator for the scalar wave equation $(\d_t^2-c^2\Delta)u=0$, but only outside $\Omega$ and only for initial data supported outside $\Omega$. Using our recently developed scattering control method, we prove that piecewise smooth wave speeds are uniquely determined by this map, and provide a reconstruction formula. In other words, the wave imaging problem is solvable in the piecewise smooth setting under mild conditions. We also illustrate a separate method, likewise constructive, for recovering the locations of interfaces in broken geodesic normal coordinates using scattering control.
\end{abstract}

\section{Introduction and background}			\label{s:intro}

\blfootnote{\!\!$^{\text{*}}$ \url{pac5@rice.edu}\qquad $^{\text{\textdagger}}$ \url{mdehoop@rice.edu} \qquad $^{\text{\textdaggerdbl}}$ \url{vk17@rice.edu} \qquad $^\parallel$ \url{gunther@math.washington.edu}}
\blfootnote{\!\!$^{\text{*}}$ $^{\text{\textdagger}}$ $^{\text{\textdaggerdbl}}$ Department of Computational and Applied Mathematics, Rice University.}
\blfootnote{\!\!$^{\text{\textdagger}}$ Department of Earth, Environmental, and Planetary Sciences, Rice University.}
\blfootnote{\!\!$^\parallel$ Department of Mathematics, University of Washington and Institute for Advanced Study, Hong Kong University of Science and Technology.}

The wave inverse problem asks for the unknown coefficient(s), representing wave speeds, of a wave equation inside a domain of interest $\Omega$, given knowledge about the equation's solutions (typically on $\bdy\Omega$). Traditionally, the coefficients are smooth, and the data is the Dirichlet-to-Neumann map, or its inverse. The main questions are uniqueness and stability: can the coefficients be recovered from the Dirichlet-to-Neumann map, and is this reconstruction stable relative to perturbations in the data? In the smooth case, the uniqueness question was answered in the affirmative by Belishev \cite{Belishev1987}, using the boundary control method introduced in that same article. Logarithmic type stability estimates were proven in \cite{BLK} for a related problem for the wave equation with a smooth sound speed or metric. Using geometric optics, Stefanov and Uhlmann~\cite{SU-Stability} show H\"older type stability for the case of simple wave speeds; recently, Stefanov, Uhlmann, and Vasy~\cite{SUV-PartialStability} also proved uniqueness, H\"older stability and reconstruction under a foliation condition, utilizing their work on the local geodesic ray transform~\cite{UVasyLocalRay}. Some work has also been done on the piecewise smooth case; e.g., Hansen~\cite{Hansen91}, assuming the background speed is known. In~\cite{KLU} it is shown that from the broken scattering relation for smooth Riemannian metrics one can determine the metric. This assumes, for the case of the sound speed, a dense number of discontinuities of the speed. For more details see \cite{Hansen91} and~\cite{KLU}.

In this paper, we show uniqueness also holds for piecewise smooth wave speeds with conormal singularities, under very mild geometric conditions, using our recently developed scattering control method~\cite{SC}. We consider the particular wave equation
\begin{equation}
	(\d_t^2 - c^2\Delta)u=0.
	\label{e:wave-eqn}
\end{equation}
Instead of using the Dirichlet-to-Neumann map, we take a slightly different (but equivalent) initial value approach: the domain is extended from $\Omega$ to $\RR^n$ and the data is the solution operator for the wave equation, but only outside $\Omega$, and only for initial data supported outside $\Omega$. The idea behind scattering control is to find, for given initial data, extra trailing initial data which allows us, indirectly, to isolate the portion of the wave field at a certain time and depth inside $\Omega$. Under appropriate geometric conditions, this portion of the wave field is free of multiple reflections arising from discontinuities in $c$ and is spatially concentrated.

These two properties of scattering control lead immediately to two strategies for the inverse problem.
The first property, spatial concentration, leads to a constructive uniqueness result inspired by a harmonic function reconstruction method of Belishev and Blagovestchenskii~\cite{Belishev97}. The key idea is to take inner products of increasingly concentrated wave fields with Euclidean coordinate functions (which are stationary for equation~\eqref{e:wave-eqn}): this allows us to convert boundary normal coordinates to Euclidean coordinates, from which $c$ can be recovered. This leads to Theorem~\ref{t:bnc-to-Euclidean}, which provides a reconstruction formula for $c$ in terms of a function $\kappa$ that can be computed by scattering control. The precise statement of the theorem involves several technical definitions which in the interests of brevity we will defer to later sections. 

Before stating the theorems, let us describe our given data, which comes in the form of an \emph{outside measurement operator} akin to the Dirichlet-to-Neumann map. Precisely, for Cauchy data $h_0 \in H^1(\RR^n) \oplus L^2(\RR^n)$, we denote by $u_{h_0}(t)$ the wave solution with initial data $h_0$. We then define the \emph{outside measurement operator} $\mathcal F: H_c^1(\Omega^c) \oplus L_c^2(\Omega^c) \to C^1(\RR_t; H^1(\Omega^c))$ as
\begin{equation}
\mathcal F\colon h_0 \to u_{h_0}(t)\big|_{\Omega^c}.
\end{equation}


\noindent We can now present a high-level version of the first main reconstruction theorem.

\begin{repmainthm}{t:bnc-to-Euclidean}
Let $y=(y^1,\dotsc,y^n)\in\Omega$ be a regular point, and let $(p,T)\in\bdy\Omega\times\RR_+$ be boundary normal coordinates for $y$. Then the Euclidean coordinates $y$ and the wave speed $c(y)$ may be reconstructed from $(p,T)$ and $\mathcal F$.
\end{repmainthm}

This reconstruction theorem carries one significant geometric restriction, the \emph{regular point} requirement on $y$. While regularity is defined rigorously in Section~\ref{s:bgnc}, the key obstruction occurs when the fastest path from $y$ to the boundary travels along an interface. Some quite reasonable choices of $(\Omega,c)$ feature an open set of such irregular points, on which Theorem~\ref{t:bnc-to-Euclidean} cannot immediately reconstruct $c$. Fortunately, a layer stripping-type argument allows us to recover such a $c$ in multiple steps, leading to an unconditional uniqueness result.

\begin{repmainthm}{t:complete-uniqueness}
	A piecewise smooth $c$ satisfying the (mild) conditions of Section~\ref{s:bgnc} is uniquely determined by $\mathcal F$.
\end{repmainthm}

The second property, multiple reflection removal, leads to a method for locating discontinuities in $c$ in (suitably generalized) boundary normal coordinates. Briefly, we may probe $\Omega$ with a wave packet and track the kinetic energy along the transmitted ray as time increases. At each discontinuity, energy is lost to the reflected wave; by measuring this loss we recover the reflection coefficient, and the time of the loss provides the depth of the discontinuity, in generalized boundary normal coordinates. Both calculated quantities (depth and reflection coefficient) become exact in the high-frequency limit.


\begin{repmainthm}{c:interface-location}
Let $\gamma(s)$ be a unit speed distance minimizing broken geodesic segment connecting a regular point $y\in\Omega$ to $\bdy\Omega$, with $\gamma(0)\in\bdy\Omega$, $\gamma(T)=y$.
Then the discontinuities of $c$ along $\gamma$, measured in boundary normal coordinates, may be reconstructed from $\mathcal F$.
\end{repmainthm}


We begin in Section~\ref{s:sc} with a re-introduction of scattering control and accompanying definitions. Section~\ref{s:harmonic} then presents the harmonic inner product-based reconstruction formula and uniqueness theorem. Sections~\ref{s:asymptotic} and~\ref{s:interface-recovery} conclude by presenting the wave packet approach to locating discontinuities in $c$.

\paragraph{Notation and Conventions}

We use $\eqml$ to indicate equality of distributions modulo smooth functions; throughout \emph{smooth} means $C^\infty$. We will extensively use Fourier integral operators associated with canonical graphs, abbreviating them as \emph{graph FIOs}. 

\section{Scattering control}								\label{s:sc}

This section revisits scattering control~\cite{SC}, a type of time-reversal iteration. Time reversal is a common theme in wave equation inverse problems, both in the mathematical literature and in practice (e.g.,~\cite{FMT}).
We present most of the key definitions and results that will be of use in the current paper.


\subsection{Domains and wave speeds}						\label{s:geometric-setup}

Let $c(x)$ be a piecewise smooth function on $\RR^n$, the wave speed, satisfying $c,c^{-1} \in L^\infty(\RR^n)$. We imagine $c$ to be known only outside a Lipschitz domain $\Omega\subset\clsr\Omega\subsetneq\RR^n$ representing the object of interest.

We allow ourselves to probe $\Omega$ with Cauchy data concentrated close to $\Omega$, in some Lipschitz domain $\Theta\supset\Omega$. We will add to this initial pulse a Cauchy data control supported outside $\Theta$, whose role is to isolate the resulting wave field at a particular time and depth controlled by a time parameter $T\in (0,\frac12\diam\Omega)$. This will require controls supported in an ambient Lipschitz neighborhood $\Upsilon$ of $\smash{\clsr\Theta}$ that satisfies $d(\bdy\Upsilon,\smash{\clsr\Theta})>2T$ and is otherwise arbitrary\footnote{Here the distance $d(x,y)$ is \emph{travel time distance}: the infimum of the lengths of all $AC$ curves $\gamma(s)$ connecting $x$ and $y$, measured in the metric $c^{-2}dx^2$, such that $\gamma^{-1}(\singsupp c)$ has measure zero; see~\eqref{e:dist-def}.}.

This initial pulse region $\Theta$ has a central role in the scattering series. First, define the \emph{depth} $d^*_\Theta(x)$ of a point $x$ inside $\Theta$:
\begin{equation}
	d^*_\Theta(x) = \when{+d(x,\bdy\Theta),		& x \in \Theta,\\
					  -d(x,\bdy\Theta),			& x \notin \Theta.}
													\label{e:exact-depth}
\end{equation}
Larger values of $d^*_\Theta$ are therefore deeper inside $\Theta$. For each $t$, define\footnote{We tacitly assume throughout that $\Theta_t$, $\Theta_t^\star$ are Lipschitz.} the open sets
\begin{nalign}
	\Theta_t^{\phantom\star} &= \set{x\in\Upsilon}{d^*_\Theta(x) > t},\\
	\Theta_t^\star &= \set{x\in\Upsilon}{d^*_\Theta(x) < t}.
													\label{e:def-Theta-t-star}
\end{nalign}
As in~\eqref{e:def-Theta-t-star} above, we use a superscript $\star$ to indicate sets and function spaces lying outside, rather than inside, some region. We define $\Omega_t$, $\Omega_t^\star$ similarly, and let $\Omega^\star=\Omega_0^\star$.

\subsection{Solution operators and spaces}					\label{s:wave-setup}

Let $\tilde{\mathbf C}$ %
%
%
 be the space of Cauchy data of interest:
\begin{align}
	\tilde{\mathbf C} = H_0^1(\Upsilon) \oplus L^2(\Upsilon),
\end{align}
considered as a Hilbert space with the \emph{energy inner product}
\begin{equation}
	\big\langle{(f_0,f_1),\,(g_0,g_1)\big\rangle} = \int_{\Upsilon} \left(\gradient f_0(x)\cdot\gradient \conj g_0(x) + c^{-2}f_1(x)\conj g_1(x)\right)\,dx.
\end{equation}
Within $\tilde{\mathbf C}$ define the subspaces of Cauchy data supported inside and outside $\Theta_{t}$:
\begin{nalign}
	\mathbf H_t &= H_0^1(\Theta_{t})\oplus L^2(\Theta_{t}),					& \hspace{1in} \mathbf H &= \mathbf H_0,\\
	\tilde{\mathbf H}_t^{\mathrlap\star} &= H_0^1(\Theta_{t}^\star) \oplus L^2(\Theta_{t}^\star),	& \tilde{\mathbf H}^{\star}\! &= \tilde{\mathbf H}_0^\star.
\end{nalign}
Define the energy and kinetic energy of Cauchy data $h=(h_0,h_1)\in\tilde{\mathbf C}$ in a subset $W\subseteq\RR^n$:
\begin{align}
	\En_W(h) &= \int_W \left(\dabs{\gradient h_0}^2 + c^{-2} \dabs{h_1}^2\right)\,dx,
	&
	\KE_W(h) &= \int_W c^{-2}\dabs{h_1}^2\,dx.
\end{align}
Next, define $F$ to be the solution operator for the initial value problem:
\begin{align}
	&F\colon H^1(\RR^n)\oplus L^2(\RR^n)\to C(\RR,H^1(\RR^n)),
	&
	&
	F(h_0,h_1) = u
	\text{\; s.t. }
	\begin{pdebracketed}
	(\partial_t^2-c^2\Delta)u &= 0,\\
				\drestr{u}_{t=0} &= h_0,\\
			\drestr{\d_t u}_{t=0} &= h_1.
	\end{pdebracketed}
	\label{e:wave-ivp}
\end{align}
Our data for the inverse problem is the \emph{outside measurement operator} $\mathcal F\colon H^1_0(\Omega^\star)\oplus L^2(\Omega^\star)\to C(\RR, H^1(\Omega^\star))$, the restriction of $F$ to $\Omega^\star$ in both domain and codomain.

Let $R_s$ propagate Cauchy data at time $t=0$ to Cauchy data at $t=s$:
\begin{align}
	R_s = \left(F,\d_t F\right)\!\Big|_{t=s}
	\mspace{-8mu}
	\colon H^1(\RR^n)\oplus L^2(\RR^n)\to H^1(\RR^n)\oplus L^2(\RR^n).
\end{align}
Now combine $R_s$ with a time-reversal operator $\nu\colon \tilde{\mathbf C}\to\tilde{\mathbf C}$, defining for a given $T$
\begin{align}
	R &= \nu\circ R_{2T},
	&
	\nu&\colon (f_0,f_1)\mapsto(f_0,-f_1).
\end{align}
In our problem, only waves interacting with $(\Omega,c)$ in time $2T$ are of interest. Consequently, let us ignore Cauchy data not interacting with $\Theta$, as follows.

Let $\mathbf G=\tilde{\mathbf H}^\star\cap\big( R_{2T}(H^1_0(\RR^n\setminus\clsr\Theta)\oplus L^2(\RR^n\setminus\clsr\Theta))\big)$ be the space of Cauchy data in $\tilde{\mathbf C}$ whose wave fields vanish on $\Theta$ at $t=0$ and $t=2T$. Let $\mathbf C$ be its orthogonal complement inside $\tilde{\mathbf C}$, and ${\mathbf H}_t^\star$ its orthogonal complement inside $\tilde{\mathbf H}_t^\star$. With this definition, $R_{2T}$ maps $\mathbf C$ to itself isometrically. Also, let $\pi_{\mathbf C}:\tilde{\mathbf C}\to\mathbf C$ be the corresponding orthogonal projection.

%
%

\subsection{Projections inside and outside $\Theta_t$}			\label{s:projections-setup}

The final ingredients needed are restriction operators for Cauchy data inside and outside each $\Theta_t$. As hard cutoffs are not bounded operators in energy space, we replace them with Hilbert space projections. 

Let $\pi_t$, $\pi_t^\star$ be the orthogonal projections of $\mathbf C$ onto $\mathbf H_t$, $\mathbf H_t^\star$ respectively; let $\clsr\pi_t=1-\pi_t$. For brevity, let $\clsr\pi=\clsr\pi_0$, $\pi^\star=\pi_0^\star$. The complementary projection $I-\pi_t-\pi^\star_t$ is the orthogonal projection onto $\mathbf I_t$, the orthogonal complement to $\mathbf H_t \oplus\mathbf H_t^\star$ in $\mathbf C$.

The Dirichlet principle provides an interpretation of these projections~\cite{SC}:
\begin{equation}
	(\clsr\pi_t h)(x) = \when{h(x), 		& x\in\Theta_t,\\
					(\phi,0),		& x\in\Theta_t^\star,}
\end{equation}
where $\phi$ is the harmonic extension of $h\restrictto{\bdy\Theta_t}$ to $\Upsilon$ (with zero trace on $\bdy\Upsilon$). Similarly, $\pi^\star_t h$ is zero on $\Theta_t$, and outside $\Theta_t$ is equal to $h$, with this harmonic extension subtracted from the first component.

\subsection{Scattering control}				\label{s:sc-def}

Our major tool is a Neumann series, the \emph{scattering control series}. Given Cauchy data $h_0\in\mathbf H$, define
\begin{equation}
	h_\infty = \sum_{i=0}^\infty (\pi^\star R)^{2i} h_0,
	\label{e:neumann}
\end{equation}
We will often need its $k\mith$ partial sum $h_k=\sum_{i=0}^k(\pi^\star R)^{2i}h_0$, as well. Formally, $h_\infty$ solves the \emph{scattering control equation}
\begin{equation}
	(I-\pi^\star R\pi^\star R)h_\infty = h_0.
	\label{e:maf}
\end{equation}
As~\eqref{e:neumann} expresses, $h_\infty$ consists of $h_0$, plus a control in $\mathbf H^\star$. In general, series~\eqref{e:neumann} does not converge in $\mathbf C$, although it does converge in an appropriate weighted space~\cite[Theorem 2.3]{SC}.

The behavior of the scattering control series is intertwined with a particular portion of the wave field, the harmonic almost direct transmission, which is at time $T$ and depth at least $T$.
\begin{definition}
	The \emph{harmonic almost direct transmission} of $h_0$ at time $T$ is
	\begin{equation}
		h\DT = h\DT(h_0,T)=\clsr\pi_T R_T h_0.
		\label{e:adt-def}
	\end{equation}
\end{definition}
Referring to the earlier discussion, we see $h\DT$ is equal to $R_Th_0$ inside $\Theta_T$; outside $\Theta_T$, its first component is extended harmonically from $\bdy\Theta_T$, while the second component is extended by zero.

We now excerpt the key theorems on the scattering control series' behavior from~\cite{SC}.

\begin{theorem}
	Let $h_0\in\mathbf H$ and $T\in(0,\frac12\diam\Theta)$. Then isolating the deepest part of the wave field of $h_0$ is equivalent to summing the scattering control series:
	\begin{nalign}
			(I-\pi^\star R\pi^\star R)h_\infty = h_0
		&\iff
			R_{-T} \clsr\pi R_{2T}h_\infty = h\DT
			\text{ and }
			h_\infty \in h_0+\mathbf H^\star.
		\label{e:basic-maf-behavior}
	\end{nalign}
	Such an $h_\infty$, if it exists, is unique in $\mathbf C$.
	\label{t:basic-maf}
\end{theorem}

\begin{theorem}
	With $h_0, T$ as in Theorem~\ref{t:basic-maf}, define the partial sums
	\begin{equation}
		h_k = \sum_{i=0}^k (\pi^\star R\pi^\star R)^i h_0.
		\label{e:partial-sums}
	\end{equation}
	Then the deepest part of the wave field can be (indirectly) recovered from $\{h_k\}$ regardless of convergence of the scattering control series:
	\begin{align}
		\lim_{k\to\infty} R_{-T}\clsr\pi R_{2T}h_k &=  R_T\chi h_0 = h\DT,
		&
		\norm{\clsr\pi Rh_k}&\searrow\norm{h\DT}.
		\label{e:maf-interior-limit}
	\end{align}
	The set of $h_0$ for which the scattering control series converges in $\mathbf C$ is dense in $\mathbf H$.
	\label{t:limit-maf}
\end{theorem}

Theorem~\ref{t:basic-maf} covers the situation when the scattering control series converges: the wavefield of $h_\infty$ inside $\Theta$ at $t=2T$ is equal to that generated by $h\DT$, the deepest portion of $h_0$'s wavefield, alone. This is not true of the wave field of $h_0$ itself, because other waves, including multiple reflections, will mix with $h\DT$'s wave field in general.

Theorem~\ref{t:limit-maf} describes the general case: convergence may fail, but only outside $\Theta$. Inside $\Theta$, the partial sums' wave fields at $t=2T$ do converge to $R_Th\DT$, and their energies are in fact monotonically decreasing.

By combining the theorems above with energy conservation, we may recover the energy of the harmonic almost direct transmission, as well as its kinetic component (which does not include a harmonic extension). For precise statements, see~\cite[Props.\ 2.7, 2.8]{SC}.

\section{Uniqueness and reconstruction of $c$ by harmonic inner products}   \label{s:harmonic}

In this section, we demonstrate how to recover $c$ by expressing it in terms of particular inner products between wave fields and harmonic functions --- inner products that can be computed by scattering control. The idea originates with Belishev and Blagovestchenskii~\cite{Belishev97} in the context of boundary control, and a similar idea was recently taken up by de Hoop, Kepley, and Oksanen and realized computationally~\cite{DKO2}. Here, the use of Cauchy data considerably simplifies the reconstruction formulas. We will restrict ourselves to piecewise smooth $c$, in order to analyze the behavior of wave fields near their wave fronts with microlocal machinery, but we expect the method is applicable to any $c$ satisfying unique continuation. 

We begin by introducing \emph{broken geodesic normal coordinates}, the natural analogue of geodesic normal coordinates for piecewise smooth metrics, in Section~\ref{s:bgnc}. Section~\ref{s:c-recovery} follows with the main theorem on recovering wave speeds with harmonic inner products. Due to the possibility of coordinate breakdown, we may not be able to recover $c$ everywhere in one pass, but prove in Section~\ref{s:layer-stripping} with a layer stripping-type argument that $c$ can be recovered on all of $\Omega$ nonetheless.

\subsection{Broken geodesic normal coordinates}     \label{s:bgnc}

Assume $\Omega\subset\RR^n$ is an open domain whose closure is an embedded submanifold with boundary in $\RR^n$. Let $c(x)$ be a piecewise smooth and lower semicontinuous function on $\RR^n$, bounded above and away from zero, and singular only on a set of disjoint, closed, connected, smooth hypersurfaces $\Gamma_i$ of $\clsr\Omega$, called \emph{interfaces}. Let $\Gamma=\bigcup\Gamma_i$; let $\{\Omega_j\}$ be the connected components of $\RR^n\setminus\Gamma$. Assume each smooth piece $\restr{c}_{\Omega_j}$ extends to a smooth function $c_j$ on $\RR^n$. 

The distance $d(X,Y)$ between sets $X,Y\subset\clsr\Omega$ is the infimal length of absolutely continuous paths $\gamma$ between points in $X$ and $Y$: 
\begin{equation}
	d(X,Y) = \inf\set*{\int c(\gamma(s))^{-2} \abs{\gamma'(s)}\,ds}%
				{\gamma\in\AC(\clsr\Omega),\, \gamma(0)\in X,\, \gamma(1)\in Y}.
	\label{e:dist-def}
\end{equation}
%
The Arzel\`a-Ascoli theorem implies that the infimum in~\eqref{e:dist-def} is always attained for closed, nonempty $X,Y$.
Under some regularity conditions, we can now identify an interior point $x$ with the closest boundary point $p(x)$ and the distance $T(x)$ between them.

\begin{definition}
The curve $\gamma\in\AC(\clsr\Omega)$ is \emph{demi-tangent} to $\Gamma$ at $\gamma(s)$ if at least one of the one-sided derivatives of $\gamma$ exists at $s$ and belongs to $T\Gamma$.

We call $x\in\Omega\setminus\Gamma$ \emph{almost regular} with respect to $(\Omega,c)$ if the infimum in $d(x,\bdy\Omega) = d(\{x\},\bdy\Omega)$ is achieved by a unique path $\gamma_x$, and this path is nowhere demi-tangent to $\Gamma\cup\bdy\Omega$. 

Let $p(x)=\gamma_x(1)$ be the closest boundary point to $x$, and $T(x)=d(x,\bdy\Omega)$. The pair $(p(x),T(x))$ are the \emph{broken geodesic normal coordinates} for $x$.
\end{definition}

The following lemma explains the name ``broken geodesic normal coordinates'' for $(p(x),T(x))$.

\begin{lemma}
	For every almost regular $x$, the minimal path $\gamma_x$ is a purely transmitted (broken) geodesic intersecting $\bdy\Omega$ normally. 
	\label{l:minimal-is-normal}
\end{lemma}

The proofs of this lemma and the others in this section are deferred to Section~\ref{s:bgnc-lemmas}.
For completeness, we recall the definition of broken geodesics.

\begin{definition}
	A \emph{(unit-speed) broken geodesic} in $(\Omega,c)$ is a continuous, piecewise smooth path $\gamma\colon \RR\supset I\to M$ that is a unit-speed geodesic with respect to $g=c^{-2}dx^2$ on $\Omega\setminus\Gamma$, intersects the interfaces $\Gamma$ at a discrete set of points $t_i\in I$. Furthermore, at each $t_i$ the intersection is transversal and Snell's Law is satisfied: that is, $\gamma'(t_i^-)-\gamma'(t_i^+)$ is normal to $\Gamma$. We will usually drop ``unit speed'' for brevity.
	
	A \emph{transmitted (broken) geodesic} in a unit-speed broken geodesic experiencing only refractions; that is, the inner products of $\gamma'(t_i^-)$ and $\gamma'(t_i^+)$ with the normal to $\Gamma$ have identical signs at each $t_i$.
\end{definition}

For every $(x,v)\in S\clsr\Omega$ there is a maximal transmitted broken geodesic $\gamma_{x,v}$ with $\gamma_{x,v}'(0)=(x,v)$. Hence the broken exponential map
\begin{equation}
\exp_{\bdy\Omega}(p,T) = \gamma_{p,\nu(p)}(T)
\end{equation}
is a left inverse for $x\mapsto(p(x),T(x))$; here $\nu(p)$ is the inward unit normal to $\bdy\Omega$ at $p$.

In the case of smooth $c$, boundary normal coordinates parametrize $\Omega$ on the complement of its cut locus. A similar result is true for broken geodesic normal coordinates:

\begin{definition}
Let $x\in\Omega$ be almost regular. Then $x$ is \emph{regular} if $d\exp_{\bdy\Omega}$ is bijective at $(p(x),T(x))$; otherwise, it is a \emph{focal point}.
Let $\Omega_r$ be the set of regular $x$.
\end{definition}


\begin{lemma}
	$\Omega_r$ is open; the broken geodesic normal coordinate map $x\mapsto (p(x),T(x))$ is a diffeomorphism between $\Omega_r$ and its image.
	\label{l:bgnc-smooth}
\end{lemma}

The significance of regular points are that these are the points where $c$ can be directly reconstructed, leading to the following property:

\begin{definition}
	$(\Omega,c)$ is \emph{totally regular} if almost every $x\in\Omega$ is regular.
	\label{d:totally-regular}
\end{definition}

\begin{figure}
	\centering
	\includegraphics{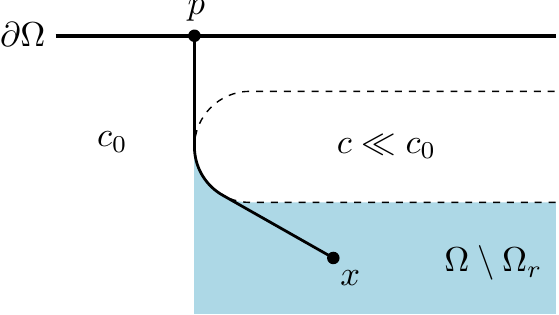}
	\caption{Failure of total regularity. Here $c$ is piecewise constant, equal to $c_0$ except on an open domain $D$ where $c\ll c_0$ (dashed). The minimal-length path from the boundary to any point $x$ in the shaded region contains a non-trivial portion of $\bdy D$.}
	\label{f:bgnc-failure}
\end{figure}

Unlike the case for smooth $c$, many reasonable choices of $(\Omega,c)$ are not totally regular. As Figure~\ref{f:bgnc-failure} illustrates, broken geodesic normal coordinates can fail to cover all of $\Omega$. In this example a single $p$ is the closest boundary point to every point in an open subset of $\Omega$. On the metric side, this occurs when minimal length paths travel along interfaces, a case we specifically excluded earlier. Conversely, if the interfaces are all strictly convex (viewed from the inside), paths along interfaces are never minimal, and in fact, $(\Omega,c)$ must be totally regular.

\begin{lemma}
	If $\Omega$ is compact and the interfaces $\Gamma_i$ are strictly convex, as viewed from their interiors, then $(\Omega,c)$ is totally regular.
	\label{l:convex-ae-regular}
\end{lemma}

\subsection{Wave speed recovery}           \label{s:c-recovery}

In the boundary control method, the \emph{Blagovestchenskii identity} allows the computation of inner products between wave fields generated by boundary controls, given only the Neumann-to-Dirichlet map. A similar identity calculates inner products between a wave field and a harmonic function. Because wave propagation is a unitary map (energy-conserving), the Blagovestchenskii identity's analogue for Cauchy data is simply the usual energy inner product. Finding inner products with harmonic functions requires only slightly more work, and relies on the fact that the wave equation~\eqref{e:wave-eqn} preserves harmonic functions.

\begin{lemma}
	For any $h_0\in\mathbf C$ and any harmonic functions $f,g$,
	\begin{align}
	\form{\overline\pi_T R_Th_0, (f,g)} &= \lim_{k\to\infty} \big[\form[\big]{h_k, (f-Tg,g)} - \form[\big]{\pi^\star R_{2T}h_k, (f+Tg,g)}\big].
	\end{align}
	If the scattering control series converges, $h_k$ can be replaced above by $h_\infty$ and the limit omitted.
	\label{l:harmonic-ip}
\end{lemma}

\begin{proof}
We begin with the observation that $f(x)+tg(x)$ is a solution of the wave equation~\eqref{e:wave-eqn} for any $c$ whenever $f,g$ are harmonic. Defining $h_k$ as before, recall from Theorem~\ref{t:limit-maf} that
\begin{equation}
	\lim_{k\to\infty} \overline\pi R_{2T}h_k = R_T \overline\pi_T R_Th_0.
\end{equation}
As a result, it is possible to compute inner products of $\overline\pi_T R_Th_0$, for arbitrary $h_0$, with arbitrary harmonic Cauchy data $(f,g)$. Namely,
\begin{nalign}
	\form{\overline\pi_T R_Th_0, (f,g)} &= \form{R_T \overline\pi_T R_Th_0, (f+Tg,g)}\\
		&= \lim_{k\to\infty} \form{\overline\pi R_{2T}h_k, (f+Tg,g)}.\\
		&= \lim_{k\to\infty} \left[\form{R_{2T}h_k, (f+Tg,g)} - \form{\pi^\star R_{2T}h_k, (f+Tg,g)}\right].
\end{nalign}
The second term is already computable from outside data. For the first term, we can move the inner product back by time $2T$ (by unitarity of $R_{-2T}$ with respect to the energy norm). Since $R_{-2T}(f+Tg,g)=(f-Tg,g)$,
\begin{nalign}
		\form{\overline\pi_T R_Th_0, (f,g)} 
		&= \lim_{k\to\infty} \left[\form{h_k, (f-Tg,g)} - \form{\pi^\star R_{2T}h_k, (f+Tg,g)}\right].
\end{nalign}
When the scattering control series converges, the limit in $k$ can be taken inside.
\end{proof}
\begin{figure}
	\centering
	\includegraphics{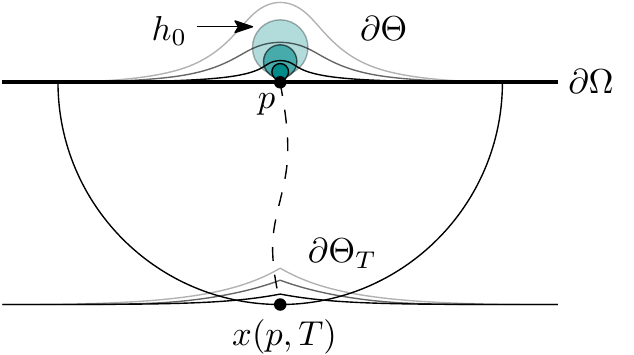}
	\caption{Shrinking the support of the initial data $h_0$ to a point. The dashed line indicates the normal geodesic from that point; the support of the almost direct transmission shrinks to a point on the geodesic.}
	\label{f:adt-shrink}
\end{figure}
The appeal of the lemma is that the almost direct transmission $\overline\pi_T R_Th_0$ in general may be arbitrarily spatially concentrated (aside from harmonic extensions in the first component). Taking inner products with the harmonic data $(0,x^i)$ and $(0,1)$, we may now recover weighted averages of $x^i$ over this support. As long as $\overline\pi_T R_Th_0$ is not oscillatory, this provides us with approximate Euclidean coordinates for the support, becoming exact in the limit as $\Theta\to\Omega$. By appropriately choosing $h_0$ and a sequence of domains $\Theta^{(j)}$ tending to $\Omega$, Euclidean coordinates $x(p,T)$ can be obtained for any point $(p,T)\in\Omega_r$ in broken geodesic normal coordinates~(Figure~\ref{f:adt-shrink}), yielding a coordinate transformation $\Phi\colon(p,T)\mapsto x(p,T)$. Once this coordinate transformation is known, $c$ can be recovered immediately by taking a derivative in $T$.

\begin{mainthm}
Let $y=(y^1,\dotsc,y^n)\in\Omega_r$, $p=p(y)\in\bdy\Omega$, and $T=d(y,\bdy\Omega)$; let $x^i$ denote the $i\mith$ Euclidean coordinate function. 
Choose a nested sequence of Lipschitz domains $\Theta^{(1)}\supset\Theta^{(2)}\supset\dotsb\supset\Omega$ such that $\smash{\bigcap_j}\Theta^{(j)}=\Omega\cup\{p\}$ and $\diam\Theta^{(j)}\setminus\Omega\to 0$. Then
\begin{align}
	y^i &= \Phi^i(p,T) = \lim_{j\to\infty} \frac{\kappa(\indicator_{\Theta^{(j)}\setminus\Omega}, x^i)}{\kappa(\indicator_{\Theta^{(j)}\setminus\Omega},1)},
	\label{e:coord-formula}
\end{align}
where $\kappa(g,f) = \form{\overline\pi_TR_T(0,\pi_{\mathbf C}g),\,(0,f)}$, and $\indicator_X$ represents the indicator function of $X$. Finally,
\begin{equation}
	c = \abs*{\dmp \Phi T}.
\end{equation}
\label{t:bnc-to-Euclidean}
\end{mainthm}

Theorem~\ref{t:bnc-to-Euclidean} solves the inverse problem and provides a reconstruction formula for $c$ in Euclidean coordinates, since $\kappa$ can be computed from outside measurements using the scattering control series (Lemma~\ref{l:harmonic-ip}). Note that $\kappa$ and $\overline\pi$ in the statement of the theorem depend on $\Theta^{(j)}$ implicitly. Uniqueness follows immediately:

\begin{corollary}
	Assume that $(\Omega,c)$ are such that $\Omega_r$ is dense in $\Omega$. Then $c$ is uniquely determined on $\Omega_T^\star$ by $R_{2T}\big|_{\Omega^\star}$.
	\label{c:harmonic-uniqueness}
\end{corollary}

\begin{remark}
	In their work, de Hoop, Kepley, and Oksanen~\cite{DKO,DKO2} use boundary controls supported on appropriate subsets of the boundary for certain time intervals, analogous to our $\Theta^{(j)}$ and $\Omega$. Their boundary controls produce \emph{wave caps} with supports similar to that of the almost direct transmission. While no formal link has yet been established between the two approaches, they are evidently closely related.
\end{remark}

\begin{remark}
	The harmonic function approach is relatively insensitive to the structure of $c$, including the locations of interfaces, if any. After recovering $c$, the Euclidean coordinates of the interfaces can be found by directly examining the reconstructed $c$.
\end{remark}

We start by stating an unsurprising but useful result about the behavior of solutions near the boundary of their domain of influence. As we do not know of a proof of it in the literature (for piecewise smooth $c$), we prove it here.
 
\begin{proposition}
	Let $(p,T)\in\Omega_r$, and let $v\in C_c^\infty(\Upsilon)$, $v(p)=1$. Then there exists a neighborhood of $(p,T)$ on which
	\begin{equation}
		R_T(0,\indicator_{\Omega_j^\star} v)(q,s) = a_0(q)H(s-T) + a_1(q)(s-T)_+
		\label{e:behavior-near-wf}
	\end{equation}
	for some nonzero $C^\infty$ function $a_0$ and a bounded function $a_1$.
	\label{p:behavior-near-wf}
\end{proposition}

\noindent Essentially, $a_0$ is the principal symbol of the purely transmitted graph FIO component of $R_T$.

\begin{proof}
To prove the lemma, we consider the initial data $h_0=(0,\indicator_{\smash{\Omega_j^\star}} v)$ as a conormal distribution on $\bdy\Omega$, apply the FIO composition calculus, then recover the progressive wave expansion~\eqref{e:behavior-near-wf} from the symbol of the resulting (polyhomogeneous) conormal distribution. As an alternative route, it may be possible to use Weinstein's principal symbols for arbitrary distributions~\cite{Wein} to allow for more general initial data.

As in section~\ref{s:wave-ml}, Cauchy data $h_0$ splits into forward- and backward-moving components $g_\pm=\pm\frac i2B^{-1}\indicator_{\Theta_J}$, and $R_Th_0=R_T^+g_++R_T^-g_-$ for half-wave solution operators $R_T^\pm$ which are order-0 FIOs away from glancing. Conjugating $(\d_t+iB)R_T^+g_+=0$, we have $R_T^-g_- = \smash{\conj{R_T^+g_+}}$ and hence $R_Th_0\eqml 2\Re R_T^+g_+$. As in the proof of Theorem~\ref{t:energy-ps-recovery}, the assumption $(p,T)\in\Omega_r$ will imply $R_T^+g_+\eqml\RDT_T^+g_+$ in a neighborhood of $x=\exp_{\bdy\Omega}(p,T)$, where $\RDT_T^+$ is the directly transmitted graph FIO component of $R_T^+$, defined as in~\eqref{e:DT-op-def}.


Let $Z=\RR^0$ be a one-point space, and define a Fourier integral operator $u\in I(Z\to\Upsilon)$ by $u(a)=ag_+$. Then $R_T^+\circ u$ is well-defined as a Fourier integral operator (distribution). In broken boundary normal coordinates (relative to $\bdy\Omega$) the initial wavefront set is $\WF u=\set{(p,0;0,\sigma)}{p\in\bdy\Omega\cap\supp v}$. The canonical relation of $\RDT_T^+$, given by the purely transmitted geodesic flow, acts as translation by $(0,T)$ on the downward ($\sigma>0$) covectors in $\WF u$, mapping them into the conormal bundle of $\Omega_T$. The images of the upward covectors in $\WF u$ are distinct from $x$, so by compactness and continuity of the flow they are bounded away from $x$; similarly for the images of $\WF u$ under the remaining graph FIO components of $R_T^+$. Hence $R_Th_0 = 2\Re R_T^+g_+ \eqml u$ on a neighborhood of $x$ for some conormal distribution $u\in I^0(\Upsilon,\bdy\Omega_T)$. Note that Lemma~\ref{l:bgnc-smooth} implies $\bdy\Omega_T$ is smooth near $x$.

Applying~\cite[Theorem 18.2.8]{H3}, write
\begin{equation}
	u(q,s+T) = \int e^{-is\sigma} a(q,\sigma)\,d\sigma,
\end{equation}
for some symbol $a\in S^{-1}$. Since $h_0$ has a homogeneous symbol as a conormal distribution, $a$ is polyhomogeneous, allowing us to write
\begin{equation}
	a=a_0(q)\sigma^{-1} + a_1(q,\sigma)\sigma^{-2},
\end{equation}
where $a_1$ is bounded in $\sigma$, and $a_0$ given by the (nonzero) principal symbol of $\RDT_T^+$. Hence
\begin{equation}
	u(q,s) = a_0(q)H(s-T) - \fiv2a_0(q) + A_1(q,s-T),
\end{equation}
with $A_1\in H^1(\RR; C^\infty(\bdy\Omega_T))\subset C^0(\RR; C^\infty(\bdy\Omega_T))$. By finite speed of propagation $(R_Th_0)(q,s) = 0$ for $s>T$, implying~\eqref{e:behavior-near-wf}.
\end{proof}

With Proposition~\ref{p:behavior-near-wf} proved, the proof of Theorem~\ref{t:bnc-to-Euclidean} is straightforward.

\begin{proof}[Proof of Theorem~\ref{t:bnc-to-Euclidean}]
Choose a bump function $v\in C_c^\infty(\Upsilon)$ equal to 1 at $p$. For all sufficiently large $j$, the assumption $\diam\Theta^{(j)}\setminus\Omega\to 0$ implies $v\indicator_{\Omega^\star}=\indicator_{\Theta^{(j)}\setminus\Omega}$ outside $\Theta^{(j)}$, so by finite speed of propagation they lead to identical almost direct transmissions: $\overline\pi_T R_T (0,v\indicator_{\Omega^\star})=\overline\pi_T R_T (0,\indicator_{\Theta^{(j)}\setminus\Omega})$. Hence, by Proposition~\ref{p:behavior-near-wf}, $R_T (0,\indicator_{\Theta^{(j)}\setminus\Omega})$ is everywhere positive or everywhere negative on the intersection of $\Omega_T^\star$ with some neighborhood $U$ of $y$.

Next, we show $\diam\Theta^{(j)}_T\setminus\Omega^{}_T\to 0$. Let $y'=(p',s'),\,y''=(p'',s'')\in \Theta^{(j)}_T\setminus\Omega^{}_T$, and let $\eps_j = \diam\Theta^{(j)}\setminus\Omega^{}$. Since $\bdy\Omega\setminus B_{\eps_j}(p)\subset \bdy\Theta^{(j)}$ we are assured $\abs{p'-p''}<2\eps_j$. Next, for all $\delta>0$ there is a piecewise $C^1$ curve $\gamma$ of length $s+\delta$ connecting $y'$ to $\bdy\Omega$, and $d(\bdy\Omega,\bdy\Theta^{(j)})<\eps_j$. Hence $T>d(y',\bdy\Theta^{(j)})>s+\delta+\eps_j$, implying $s\in(T-\eps_j,T)$. Hence $\abs{s'-s''}<2\eps_j$, and we can conclude $\smash{\diam\Theta^{(j)}_T\setminus\Omega^{}_T}<C\eps_j\to 0$ for some constant $C$.

In particular, $\smash{\Theta^{(j)}_T\setminus\Omega^{}_T}$, which contains the support of the second component of $\clsr\pi_T R_T (0,\indicator_{\Theta^{(j)}\setminus\Omega})$, lies in $U\cap\Omega_T^\star$ for large enough $j$. For such $j$,
\begin{equation}
	\inf_{\mathclap{\Theta^{(j)}_T\setminus\Omega^{}_T}} x^i \leq \frac{\kappa(\indicator_{\Theta^{(j)}\setminus\Omega}, x^i)}{\kappa(\indicator_{\Theta^{(j)}\setminus\Omega},1)}
	\leq
	\sup_{\mathclap{\Theta^{(j)}_T\setminus\Omega^{}_T}} x^i.
\end{equation}
Note $\kappa$ is finite, by compactness of $\Theta^{(j)}_T\setminus\Omega^{}_T$ and the boundedness of $R_T(0,\indicator_{\Theta^{(j)}\setminus\Omega})$. As $\diam\Theta^{(j)}_T\setminus\Omega^{}_T\to 0$, the infimum and supremum above tend to the Euclidean coordinate $y^i$, completing the proof.
\end{proof}

\subsection{Uniqueness and layer stripping}                          \label{s:layer-stripping}

In this section, we combine Theorem~\ref{t:bnc-to-Euclidean} with a layer-stripping argument to prove uniqueness for all piecewise smooth $c$ with conormal singularities, even when $(\Omega,c)$ is not totally regular.

\begin{mainthm}
	$c$ is uniquely determined by $\mathcal F$.
	\label{t:complete-uniqueness}
\end{mainthm}

The idea is as follows. With Theorem~\ref{t:bnc-to-Euclidean} we may not always be able to reconstruct $c$ everywhere, but we can always do so in a neighborhood of the boundary, where broken geodesic normal coordinates exist. We may then shrink the boundary inward, into the region where $c$ is now known, and by choosing the new boundary $\bdy\tilde\Omega$ suitably, reveal more regular points where we may reconstruct $c$ with Theorem~\ref{t:bnc-to-Euclidean}. By repeating this process, we can show that $c$ can be reconstructed everywhere. 

\begin{proof}
For the proof, we choose a point $x$ on the boundary of the domain where $c$ is known, and pick a new boundary $\bdy\tilde\Omega$ constructed to have a unique point closest to $x$, as well as to all points on a geodesic segment containing $x$. By Theorem~\ref{t:bnc-to-Euclidean}, $c$ can be then be reconstructed on this segment.

Suppose $c$, $c'$ are two piecewise smooth functions, bounded and bounded away from zero, equal outside $\Omega$ with singular supports $\Gamma$, $\Gamma'$ that are disjoint unions of smooth hypersurfaces. Let $\mathcal F$, $\mathcal F'$ be the corresponding outside measurement operators.

Assume $\mathcal F=\mathcal F'$, and let $O=\{c(x)\neq c'(x)\}\setminus (\Gamma\cup\Gamma')$, which is open in $\Omega$. We would like to show that $O$ is empty. Suppose the contrary and choose some $x\in\bdy O$.
Choose a covector $\xi$ that points out of $O$, and is not tangent to $\Gamma$ (if $x\in\Gamma$). Let $\gamma$ be the geodesic emanating from $(x,\xi)$; and choose $\eps>0$ such that $\gamma|_{[-\eps,\eps]}$ does not intersect $\Gamma$ except possibly at $x$. Set $y=\gamma(-\eps)$ and $z=\gamma(\eps)$, and let $P=B_{2\eps}(y)\owns z$.

Now, choose a Lipschitz subdomain $\tilde\Omega\supset P\cup O$ intersecting $\bdy P$ at $z$ only.
By construction, $\gamma$ is the unique distance-minimizing path from $y$ to $\bdy\tilde\Omega$, and it is smooth and transversal to $\Gamma$. It follows immediately that $\gamma$ is also the only distance-minimizing path from $\gamma(s)$ to $\bdy\tilde\Omega$ for $s\in(-\eps,\eps)$. 
Finally, Lemma~\ref{l:past-focal} implies that there are no focal points on $\gamma|_{(-\eps,\eps)}$, so every point in $\gamma((-\eps,\eps))$ is regular. The same argument holds for $c'$; so, shrinking $\epsilon$ if necessary, the points in $\gamma((-\eps,\eps))\setminus(\Gamma\cup\Gamma')$ are regular with respect to both $(\tilde\Omega,c)$ and $(\tilde\Omega,c')$.

By Lemma~\ref{l:redatum}, the outside measurement operators $\tilde{\mathcal F}$, $\tilde{\mathcal F}'$ for $c, c'$ are identical. Applying Theorem~\ref{t:bnc-to-Euclidean}, we find $c=c'$ on $\gamma((-\eps,\eps))$, a contradiction. 
\end{proof}

In the proof of Theorem~\ref{t:complete-uniqueness}, we used the fact that the outside measurement operator for a smaller domain $\tilde\Omega\subset\Omega$ can be found from $\mathcal F$ if we know the wave speed between $\bdy\tilde\Omega$ and $\Omega$. The following lemma provides the details.


\begin{lemma}
	Let $\tilde\Omega\subset\Omega$, and let $\tilde{\mathcal F}$, $\mathcal F$ be the corresponding outside measurement operators. Then $\tilde{\mathcal F}$ is uniquely determined by $\mathcal F$ and $c\big|_{\RR^n\setminus\tilde\Omega}$. 
	\label{l:redatum}
\end{lemma}

In the boundary control setting, de Hoop, Kepley, and Oksanen consider the process of finding (an analogue of) $\tilde{\mathcal F}$ in much more detail, using the Neumann-to-Dirichlet map in place of the outside measurement operator $\mathcal F$~\cite{DKORedatum}. They also consider the problem's stability and give a concrete reconstruction procedure.

\begin{proof}
	The proof is a standard application of unique continuation. Choose $T>d_g(\bdy\Omega,\bdy\tilde\Omega)$, and consider the map
	\begin{nalign}
		G_T\colon (H^1_0(\Omega^\star)\oplus L^2(\Omega^\star))^2 &\to H^1_0(\tilde\Omega^\star)\oplus L^2(\tilde\Omega^\star)\\
					(h_+,h_-) & \mapsto R_Th_+\big|_{\tilde\Omega^\star} + R_{-T}h_-\big|_{\tilde\Omega^\star}.
	\end{nalign}
	As usual, $\Omega^\star=\RR^n\setminus\clsr\Omega$, and similarly for $\tilde\Omega^\star$.
	We would like to show approximate controllability: that is, the image of $G_T$ is dense in the codomain, or equivalently, $\ker G_T^*=0$. For any $h$ and $h_+,h_-$, by the unitarity of $R_T$,
	\begin{nalign}
		\form{G_T(h_+,h_-),h}_{\tilde\Omega^\star} &= \form{R_Th_+,h}_{\RR^n} + \form{R_{-T}h_-,h}_{\RR^n} \\
						&= \form{h_+,R_{-T}h}_{\RR^n} + \form{h_-,R_Th}_{\RR^n} \\
						&= \form{h_+,\pi^\star R_{-T}h}_{\Omega^\star} + \form{h_-,\pi^\star R_Th}_{\Omega^\star},
	\end{nalign}
	where $\norm\cdot_{X}$ is the energy inner product on $H^1_0(X)\oplus L^2(X)$. 	Hence $G_T^*h = (\pi^\star R_{-T}h, \, \pi^\star R_{T}h)$.
	
	Suppose now $h\in\ker G_T^*$, and consider the wavefield $Fh$ produced by $h$. Since $Fh(T,\cdot)$ and $Fh(-T,\cdot)$ are harmonic on $\Omega^\star$, we conclude $\d_t^2Fh(\pm T,\cdot)\big|_{\Omega^\star}=0$. By finite speed of propagation and unique continuation, $\d_t^2 Fh(0,\cdot)=0$ on the set $\Omega_T^\star=\set{x\in\Omega}{d(x,\bdy\Omega)<T}\supset\smash{\tilde\Omega}$~\cite[Lemma 2.9]{SC}. In particular, $h = Fh(0,\cdot)$ is harmonic on $\tilde\Omega^\star$, but since $h\in H^1_0$ this forces $h=0$. This proves injectivity of $G_T^*$, and hence approximate controllability of $G_T$.
	
	Now consider an arbitrary $h\in H^1_0(\tilde\Omega^\star)\oplus L^2(\tilde\Omega^\star)$; we must show $\tilde{\mathcal F}h=Fh|_{\tilde\Omega^\star}$ is determined by $\mathcal Fh=Fh|_{\Omega^\star}$ and $c|_{\tilde\Omega^\star}$. Accordingly, let $c'$ be another wave speed satisfying the conditions required of $c$; assume its outside measurement operator $\mathcal F'$ is identical to $\mathcal F$, and that $c=c'$ on $\tilde\Omega^\star$.
	
	Choose a sequence $h_i=G_T(h_{i,+}, h_{i,-})\to h$; let $u_i(t,x)=Fh_{i,+}(t+T,x)+Fh_{i,+}(t-T,x)$ be the associated wavefields with respect to $c$ and define $u_i'$, the wavefields with respect to $c'$, similarly. By continuity $u_i\to Fh$, $u_i'\to F'h$. The difference $v_i=u_i-u_i'$ is a $c$-wave equation solution outside of $\tilde\Omega$ and is zero outside $\Omega$ since $\mathcal F=\mathcal F'$. Hence, by unique continuation, $v_i=0$ outside $\tilde\Omega$, and that implies $Fh=F'h$ outside $\tilde\Omega$.
\end{proof}

\subsection{Geometric proofs} 			\label{s:bgnc-lemmas}

We conclude by proving several lemmas on broken geodesic normal coordinates from Section~\ref{s:bgnc}.

\begin{proof}[Proof of Lemma~\ref{l:minimal-is-normal}]
	To start, split $\gamma_x$ into pieces $\gamma_i$, each contained in a single domain $\Omega_{j(i)}$.
	Let $0=a_0<a_1<\dotsc<a_{r-1}<a_r=1$, where $\{a_1,\dotsc,a_{r-1}\} = \gamma_x^{-1}(\Gamma)$. Write $\gamma_i = \smash{\restr{\gamma}}_{[a_{i-1},a_i]}$, and let $\clsr\Omega_{j_i}$ be the subdomain containing $\gamma_i$.
	
	First, we show that $\gamma_x$ is a broken geodesic. Each $\gamma_i$ must be a geodesic for $(\Omega_{j(i)}, \smash{c_{j(i)}^{-2}}dx^2)$, and in particular $C^1$, for otherwise $\gamma$ could be shortened by replacing $\gamma_i$ with a distance-minimizing geodesic between $\gamma_i$'s endpoints. Snell's law holds at the interfaces as a direct consequence of the first variation formula for geodesics~\cite[Proposition 6.5]{LeeRM}. 
	
	Next, if there is a reflection ($j(i)=j(i+1)$), then $\gamma_i\cup\gamma_{i+1}$ is not the minimal-length path from $a_{i-1}$ to $a_{i+1}$ as the corner can be ``rounded''; see \cite[Theorem 6.6]{LeeRM}. 
Hence $\gamma_x$ contains only refractions. Finally, $\gamma_x'(1)=\gamma_r'(1)$ must be normal to $\bdy\Omega$, again by the first variation formula for geodesics.
	%
\end{proof}


\begin{proof}[Proof of Lemma~\ref{l:bgnc-smooth}]
	By definition, $\exp_{\bdy\Omega}$ is an injective local diffeomorphism on the interior of $\exp_{\bdy\Omega}^{-1}(\Omega_r)$, so it suffices to prove $\Omega_r$ is open.
	
	Suppose now $\Omega\setminus\Omega_r\supset \{x_j\}\to x\in\Omega$, and choose minimal-length paths $\gamma_j\colon [0,1]\to\clsr\Omega$ from each $x_j$ to $\bdy\Omega$. Using Arzel\`a-Ascoli, we may assume, by taking a subsequence, that $\gamma_j\to\gamma$ uniformly. Letting $l(\eta)=\int c(\eta(s))^{-2}\abs{\eta'(s)}\,ds$ be the length of $\eta\in\AC(\clsr\Omega)$, define $p_j=\gamma_i(1)\in\bdy\Omega$, $T_j=\ell(\gamma_i)$, and similarly $p=\gamma(1)$, $T=\ell(\gamma)$.
	
	There are four cases, depending on how the $x_j$ are irregular:
	\begin{itemize}
		\item Infinitely many $x_j\in\Gamma$: then $x\in\Gamma$, by closedness of $\Gamma$.
				
		\item For infinitely many $j$ there exist distinct minimal-length paths $\tilde\gamma_j\neq\gamma_j$ from $x_j$ to $\bdy\Omega$: As before, using Arzel\`a-Ascoli and passing to a subsequence, $\tilde\gamma_j$ converges to some minimal-length path $\tilde\gamma$ from $x$ to $\bdy\Omega$. If the minimal paths $\gamma,\tilde\gamma$ are distinct, then $x$ is irregular.
		
		Otherwise, let $p_j=\gamma_i(1)\in\bdy\Omega$, $T_j=\ell(\gamma_i)$, and similarly define $\tilde p_j$, $\tilde T_j$. The fact that $\lim (p_j,T_j)=\lim(\tilde p_j,\tilde T_j)$ and $\exp_{\bdy\Omega}(p_j,T_j)=\exp_{\bdy\Omega}(\tilde p_j,\tilde T_j)$ while $(p_j,T_j)\neq(\tilde p_j,\tilde T_j)$ implies that $d\exp_{\bdy\Omega}$ is singular at $\lim (p_j,T_j)$. Hence $x$ is irregular.
		
		\item Infinitely many $\gamma_j$ are demi-tangent to $\Gamma\cup\bdy\Omega$: First, assume infinitely many left-hand side derivatives $\gamma'^-$ are tangent to $\Gamma$. Passing to a subsequence, assume this is true for all $\gamma_j$. Let $s_j = \inf\set{s}{\gamma'^-_j(s)\in T\Gamma}$; by compactness the $s_j^*$ have a limit point $s^*\in(0,1]$. Again passing to a subsequence, we may assume $s_j^*\to s^*$, and $\gamma'(s_j^*)$ converges to some $\nu\in T\Gamma$. The proof of Lemma~\ref{l:minimal-is-normal} implies that each $\gamma_j$ is a normal transmitted geodesic on $[0,s_j^*]$, and so by the geodesic equation is $C^2$-bounded on $\Omega\setminus\Gamma$, the bounds depending on the $C^1(\Omega\setminus\Gamma)$ norm of $c$ in some neighborhood of $\gamma([0,s^*])\setminus\Gamma$.
		
		Consider $\gamma(s^*-\eps)$ for $\eps>0$. If $\gamma(s^*-\eps)$ is outside $\Gamma$, then so is $\gamma_j(s_j^*-\eps)$ for sufficiently large $j$. Write $\gamma_j(s_j^*-\eps) = \gamma_j(s_j^*)-\eps\gamma'_j(s_j^*)+\eps^2 R_j(\eps)$. Since $\gamma_j$ is a geodesic near $s_j^*-\eps$, the remarks above imply $R_j(\eps)$ is smooth and uniformly bounded in $j$. Taking limits, we have $\gamma(s^*-\eps) = \gamma(s^*)-\eps\nu + O(\eps^2)$, with a locally bounded remainder term. If $\gamma^{-1}(\Omega\setminus\Gamma)$ is dense in a neighborhood of $s^*$, this implies $\gamma'(s^*)$ exists and equals $\nu\in T\Gamma$. If not, then by continuity $\gamma\in\Gamma$ near $s^*$; hence $\gamma'$ exists and lies in $T\Gamma$ for some $s$ near $s^*$. Either way, $x$ is irregular.%
		\footnote{The argument here covers the possibility of interfaces that are smooth but extremely oscillatory.}
		
		Finally, if infinitely many right-hand side derivatives $\gamma'^+$ are tangent to $\Gamma$, a similar argument applies, flipping signs and replacing $s_j$ by the supremum $\sup\set{s}{\gamma'^+_j(s)\in T\Gamma}$.
		
		\item Lastly, if $d\exp_{\bdy\Omega}(p(x_j),T(x_j))$ is singular for infinitely many $j$, the same is true at $p(x), T(x)$ by continuity.
	\end{itemize} 
		
	It is clear that $\Omega_r$ is open in $\RR^n$, not just in $\Omega$, since boundary normal coordinates are always smooth and well-defined in a sufficiently small neighborhood of $\bdy\Omega$, and therefore none of the conditions for regularity can fail near $\bdy\Omega$.
\end{proof}

%
%
%

\begin{proof}[Proof of Lemma~\ref{l:convex-ae-regular}]
	Let $x\in\Omega$, and let $\gamma_x\in AC(\clsr\Omega)$ be a minimal-length path from $x$ to $\bdy\Omega$. 
		
	Suppose first $\gamma$ intersects $\Gamma$ at infinitely many points. Then, by continuity, $\gamma^{-1}(\Gamma)$ contains some closed interval $[a,b]$. However, by strict convexity the minimal-length path from $\gamma(a)$ to $\gamma(b)$ cannot be contained in any intersection $\rho^{-1}(\gamma(a))$, a contradiction. 

	Therefore, $\gamma$ must intersect $\Gamma$ at only finitely many points, and between these intersections it must be a geodesic, just as in Lemma~\ref{l:minimal-is-normal}. By strict convexity, $\gamma$ may intersect $\Gamma$ tangentially at most once. If it does so, say in component $\Gamma_i$, then $\gamma$ does not enter the domain $\Omega_i$ bounded by $\Gamma_i$.
	
	Our goal now is to show that among the set of boundary normal transmitted geodesics (those issued form $N^*\bdy\Omega$), almost none glance from $\Gamma$. Let $N\subset T^*\Omega$, the set of normal geodesic covectors,  be the flowout of $N^*\bdy\Omega$ by $\Phi$. Similarly, let $\mathcal G\subset T^*\Omega$, the set of eventually glancing covectors, be the flowout of $T^*\Gamma$ by the continuous extension of $\Phi$. 
	
	We now show $N\setminus \mathcal G$ is dense in $N$ by checking $T_{y,\eta}N\not\subset T_{y,\eta}\mathcal G$ at any intersection $(y,\eta)\in T^*(\Omega\setminus\Gamma)$. The idea is that every glancing normal transmitted geodesic can be perturbed downward to a non-glancing normal geodesic. Let $\gamma$ be the (normal) transmitted geodesic through $(y,\eta)$, and let $\{z\}=\gamma\cap\Gamma$; say $z=\Phi_t(y)$, for some $t\in\RR$.
	
	If $\nu$ is the inward pointing normal to the component $\bdy\Omega_i=\Gamma_i$ of $\Gamma$ at $z$, consider the points $z_\eps=z+\eps\nu$. For each $z_\eps$ there is a $\zeta_\eps$ such that $(z_\eps,\zeta_\eps)\in N$. Hence $\alpha=d\Phi_t(\nu, d\zeta_\eps/d\eps\big|_{\eps=0})\in T_{y,\eta} N$. However, $\alpha$ cannot belong to $T\mathcal G$, for if it did, we would have a perturbation $(z'_\eps,\zeta'_\eps)=(z_\eps,\zeta_\eps)+O(\eps^2)$ such that each of the transmitted bicharacteristics $\bgamma'_\eps$ through $(z'_\eps,\zeta'_\eps)$ glance from $\Gamma$. But for small enough $\eps$, this is impossible: $\bgamma'_\eps$ cannot glance from interfaces outside $\clsr\Omega_i$, because $\bgamma$ does not; it cannot glance from $\Gamma_i$ by convexity, because it intersects $\Omega_i$; and it cannot intersect interfaces inside $\Omega_i$, because they are a finite distance from $z$. 
	
	Therefore, a dense subset of points $x\in\Omega$ have minimal paths to the boundary not glancing from $\Gamma$. Next, we ensure that not too many of these are focal points or have \emph{multiple} minimal paths.
	
	For this, suppose $\gamma$ is a minimal-length admissible path from $x\in\Omega$ to $\bdy\Omega$, with $\gamma(0)=x$ and $\gamma(1)\in\bdy\Omega$. Then we can check that all of $\gamma((0,1))$ is almost regular. For suppose some $\gamma(s)$, $s\in(0,1)$, had another minimal admissible path to $\bdy\Omega$ besides $\gamma$, say $\eta$. Then the union $\eta_0=\gamma|_{[0,s]}\cup\eta$ would also be minimal-length, and therefore must be a purely transmitted geodesic, recalling the proof of Lemma~\ref{l:minimal-is-normal}. But this is impossible, for if $\gamma(s)\notin\Gamma$, then $\eta_0$ has a corner at $\gamma(s)$, while if $\gamma(s)\in\Gamma$, then $\eta_0$ cannot satisfy Snell's law at $\gamma(s)$. In particular, $x$ is the limit of a sequence of almost regular points, namely $\lim_{s\to 0}\gamma(s)$. Finally, by Lemma~\ref{l:past-focal}, broken normal geodesics do not minimize distance to the boundary past a focal point, so in fact $\gamma((0,1))\subset\Omega_r$, completing the proof.
\end{proof}

\begin{lemma}
	Let $\gamma(s)=\exp_{\bdy\Omega}(p,s)$ be a broken normal geodesic, where $p\in\bdy\Omega$. If $\gamma(s_0)$ is a focal point and $\gamma(s_0)\notin\Gamma$, then $\gamma$ does not minimize beyond $s_0$. That is, $\gamma(s')<s'$ for $s'>s_0$.
	\label{l:past-focal}
\end{lemma}
\begin{proof}
	The lemma will be proved by reducing to the smooth case, where the result is well known; e.g.~\cite[\textsection{}1.12 and (2.5.15)]{Klingenberg}. When $c$ is smooth, focal points are discrete along normal geodesics, which follows from the symplectic property of the geodesic flow as well as a twist condition (see~\cite[prop.~2.11]{PaternainGeodesics}). Since the broken geodesic flow for fixed time parameter is also described by a canonical graph, and satisfies the same twist condition, essentially the same proof shows that focal points are also discrete along broken normal geodesics.
	
	Choose then an interval $[s_1,s_2]\owns s_0$ on which $\gamma(s_0)$ is the sole focal point, and such that $\gamma\big|_{[s_1,s_2]}$ does not intersect $\Gamma$. Let $\Omega_{s_1}=\set{x\in\Omega}{d(x,\bdy\Omega)\ge s_1}$; since $\gamma(s_1)$ is not a focal point, $\bdy\Omega_{s_1}$ is a smooth hypersurface near $\gamma(s_1)$. Now $\gamma\big|_{[s_1,s_2]}$ is a normal geodesic having a focal point at $\gamma(s_0)$, with respect to $\bdy\Omega_{s_1}$. The result for the smooth case implies $\gamma$ is not minimizing (w.r.t. $\bdy\Omega_{s_1}$) past $s_0$. That is, $d(\bdy\Omega_{s_1},\gamma(s))<s-s_1$ for $s\in(s_0,s_2]$. Because $d(x,\bdy\Omega)=s_1$ for every point $x\in\bdy\Omega_{s_1}$, this implies $d(\bdy\Omega,\gamma(s))<s$ for $s\in(s_0,s_2]$. It follows immediately that $d(\bdy\Omega,\gamma(s))<s$ for all $s>s_0$, completing the proof.
\end{proof}

\section{Asymptotic Analysis}      \label{s:asymptotic}

In this section and the next, we prove a complementary result on locating the discontinuities in $c$ in boundary normal coordinates. In geophysics, this is akin to a time migration, with multiple scattering completely suppressed. Our basic procedure involves sending a wave packet into $\Omega$ and tracking its energy as it proceeds; at each discontinuity in $c$ energy will be lost to the reflected wave, which we can measure with scattering control. As before, we restrict our attention to the wave equation~\ref{e:wave-eqn}; however, the argument is expected to generalize to arbitrary scalar wave equations.

In preparation, we begin in Sections~\ref{s:wave-packets} and~\ref{s:ps-recovery} by studying how the energy of a wave packet is transformed by a graph FIO. Wave packets and wave packet frames have a long history in microlocal analysis, starting with C\'ordoba-Fefferman~\cite{CF}; see for example~\cite{S1,S2,CD,GT,ST,DUVW}. Our rather loose definition is inspired by Smith~\cite{S2}. As further preparation, we then recall in Section~\ref{s:wave-ml} the well-known decomposition of the wave equation parametrix into components involving reflections and refractions, when the wave speed is discontinuous. We conclude in Section~\ref{s:interface-recovery} with the main result.

\subsection{Wave packets and propagation of singularities}  \label{s:wave-packets}

Let $\phi$ be a Schwartz function (the \emph{standard wave packet}) satisfying
\begin{itemize}
	\item $\supp\hat\phi\subset\{1 < \xi_1\}$;
	\item $\supp\hat\phi$ compact;
	\item $\norm\phi_{L^2}=1$.
\end{itemize}
We then introduce parabolic dilates of $\phi$, given by a scale factor $\lambda$:
\begin{equation}
	\phi_{\lambda} = \lambda^{(n+1)/4}\phi(\lambda x_1, \sqrt\lambda x_2,\dotsc,\sqrt\lambda x_n).
	\label{e:parabolic-dilates}
\end{equation}
The leading power of $\lambda$ ensures that $\norm{\phi_\lambda}_{L^2}=1$.
Finally, we introduce translations and rotations as follows. For $(x,\xi)\in S^*\RR^n$, let $\phi_{\lambda,x,\xi}=\phi_\lambda\circ M_{x,\xi}$, where $M_{x,\xi}$ is a rigid motion such that $dM_{x,\xi}^*(0,e_1)=(x,\xi)$, where $e_1=(1,0,\dotsc,0)$.
The result $\phi_{\lambda,x,\xi}$ is a wave packet of frequency $\lambda$ centered at $(x,\xi)$. For brevity, we accumulate the indices into a single index $\mu=(\lambda,x,\xi)$.

Next, we describe the frequency and spatial concentration of $\phi_\mu$.
Define $\Xi_{\mu}=\cone(\supp\hat\phi_\mu)\subset\RR^n$, where $\cone(Y)=\bigcup_{a\in\RR^+}\!aY$ is the smallest conic set containing $Y$.
On the spatial side, choose neighborhoods $U_{\lambda}\owns 0$ satisfying as $\lambda\to\infty$
\begin{itemize}
	\item $\diam U_{\lambda}\to 0$;
	\item $\int_{U_{\lambda}} \abs*{\phi_{\lambda}}^2\,dx \to 1$.
\end{itemize}
Such $U_\lambda$ exist, since by~\eqref{e:parabolic-dilates} $\phi_\lambda$ becomes increasingly concentrated near the origin as $\lambda\to\infty$; we may take $U_{\lambda}=B_r(\lambda)(0)$ with radius $r(\lambda)\sim \lambda^{-1/2+\eps}$, for instance.

Next, define slightly larger sets $U'_\lambda$ satisfying the same conditions, with $\clsr U_\lambda\subset U'_\lambda$. 
Set $U_{\lambda,x,\xi}=M_{x,\xi}^{-1}(U_{\lambda,0,e_1})$, and similarly for $U'_{\lambda,x,\xi}$, and choose cutoffs $\rho_\mu$ satisfying
\begin{align}
\rho_\mu(x)&=\when{	1, & x\in \clsr U_\mu,\\
				0, & x\notin U'_\mu.}
\end{align}

This completes the construction.
Intuitively speaking, a graph FIO maps wave packets to wave packets, preserving microlocal concentration~\cite{S1,S2}. Here, we only need the fact that an FIO preserves a wave packet's spatial concentration, as expressed in the following lemma.

\begin{lemma} \label{l:post-fio-concentration}
	Let $T$ be a graph FIO of order zero with associated symplectomorphism $\chi$. Let $(x_0,\xi_0)\in S^*\Theta$, and $(y_0,\eta_0)=\chi(x_0,\xi_0)$. Then for any neighborhood $V\owns y_0$,
	\begin{equation}
		\norm*{T\rho_{\lambda,x_0,\xi_0}\phi_{\lambda,x_0,\xi_0}}_{L^2(\RR^n\setminus V)} \to 0
		\qquad
		\text{ as $\lambda\to\infty$.}
	\end{equation}
\end{lemma}
\begin{proof}
	We start by cutting off $T$ near $(x_0,\xi_0)$ and away from $y$. Choose a smooth cutoff $\sigma_y$ supported in $V$ and equal to 1 on a smaller neighborhood $V'\owns y_0$. Let $W=\chi^{-1}(V'\times \RR^n)$, and pick $W'$ with $(x_0,\xi_0)\in W'\subset\clsr{W'}\subset W$. Let $\alpha(x,\xi)$ be a smooth conic cutoff supported in $W$ and equal to one on $\clsr{W'}$, and $\sigma_xu = (2\pi)^{-n} \int e^{ix\cdot\xi}\alpha(x,\xi)\hat u(\xi)\,d\xi$ the associated pseudodifferential cutoff.
		

	For the lemma, it suffices to show $\norm{(1-\sigma_y)T\rho_\mu\phi_\mu}\to 0$. Actually, since $\abs{(1-\rho_\mu)\phi_\mu}\to 0$ as $\lambda\to\infty$, by $L^2$ boundedness of $T$ it is enough to show $\norm{(1-\sigma_y)T\phi_\mu}\to 0$. For this we split $(1-\sigma_y)T$:
\begin{align}
	(1-\sigma_y)T &= K + L,
	&
	K &= (1-\sigma_y)T\sigma_x,
	&
	L &= (1-\sigma_y)T(1-\sigma_x).
\end{align}
By definition, $L\phi_\mu=0$, since $[1-\alpha(x,\xi)]\hat\phi_\mu(\xi)$ is identically zero.
By construction, $K$ is smoothing, since its amplitude is zero on the graph of $\chi$. In particular, $K$ is continuous from $H^{-s}\to L^2$ for any $s$, so
\begin{nalign}
	\norm*{K\phi_\mu}_{L^2}^2 	&\lesssim \norm*{\phi_\mu}_{H^{-s}}^2 \\
							&= \frac{1}{(2\pi)^n}\int \form\xi^{-2s} \abs*{\hat\phi_\mu(\xi)}^2\,d\xi\\
							&\leq \form{\lambda}^{-2s} \frac{1}{(2\pi)^n}\int \abs*{\hat\phi_\mu(\xi)}^2\,d\xi\\
							&\lesssim \lambda^{-2s},
\end{nalign}
using the fact that $\abs\xi>\lambda/2$ on $\supp\hat\phi_\mu$. This completes the proof.
\end{proof}

\subsection{Recovery of principal symbols}  \label{s:ps-recovery}

With the framework laid in the previous subsection, we now show a graph FIO scales the $L^2$ norm of a wave packet by the principal symbol, to leading order.

\begin{proposition}\label{p:principal-symbol-recovery}
	Let $T$ be a graph FIO of order zero with associated symplectomorphism $\chi$, with principal symbol $p$. Let $(x_0,\xi_0)\in S^*\Theta$, and $(y_0,\eta_0)=\chi(x_0,\xi_0)$. Then for any neighborhood $V$ of $(y_0,\eta_0)$,
	\begin{equation}
		\norm*{T\rho_{\lambda,x,\xi}\phi_{\lambda,x,\xi}}^2_{L^2(V)} \to \abs*{p(x_0,\xi_0)}^2
		\qquad
		\text{ as $\lambda\to\infty$.}
	\end{equation}
\end{proposition}
\begin{proof}
	Let $p_0=p(x_0,\xi_0)$, $\mu=\mu(\lambda)=(\lambda,x_0,\xi_0)$.
	Since $(1-\rho_\mu)\phi_\mu\to 0$ and because of Lemma~\ref{l:post-fio-concentration}, it suffices to prove this limit holds with a norm on all of $\RR^n$, that is,
	\begin{equation}
		\norm*{T\phi_\mu}^2_{L^2(\RR^n)}\to \abs{p_0}^2.
	\end{equation}
	Given $\eps>0$, there exists a $\lambda_0$ such that for all $\lambda\geq\lambda_0$,
	\begin{equation}
		\underline q = \abs{p_0}^2-\eps < \abs{p(x,\xi)}^2 < \abs{p_0}^2+\eps = \overline q
		\qquad
		\text{ for } x\in \clsr{U_{\lambda,x_0,\xi_0}},\,\xi\in\clsr{\Xi_{\lambda,x_0,\xi_0}}.
	\end{equation}
	Fix $\lambda_1>\lambda_0$, and choose a smooth conic cutoff $\alpha(x,\xi)$ supported in $U_{\lambda_0,x_0,\xi_0}\times\Xi_{\lambda_0,x_0,\xi_0}$ and equal to 1 on $U_{\lambda_1,x_0,\xi_0}\times\Xi_{\lambda_1,x_0,\xi_0}$. Let $\sigma u = (2\pi)^{-n} \int e^{ix\cdot\xi}\alpha(x,\xi)\hat u(\xi)\,d\xi$ be the associated pseudodifferential cutoff.
	
	Assuming from now on $\lambda>\lambda_1$, we note $(\Id-\sigma)\phi_\mu=0$.
	Letting $Q=T^*T\in\Psi^0$,
	\begin{align}
		\norm{T\phi_\mu}_{L^2}^2 = \form{\phi_\mu,Q\phi_\mu} = \form{\phi_\mu,Q\sigma\phi_\mu}.
	\end{align}
	Applying the sharp \Garding{} inequality to $(Q-\underline qI)\sigma$ and $(\overline qI-Q)\sigma$,
	\begin{align}
		\form{\phi_\mu,Q\sigma\phi_\mu} & \geq \underline q \norm{\phi_\mu}_{L^2}^2 + \form{\phi_\mu,\underline K\phi_\mu}, \\
		\form{\phi_\mu,Q\sigma\phi_\mu} & \leq \overline q \norm{\phi_\mu}_{L^2}^2 + \form{\phi_\mu,\overline K\phi_\mu},
	\end{align}
	for smoothing operators $\underline K$, $\overline K$. Arguing as in the proof of Lemma~\ref{l:post-fio-concentration}, $\underline K\phi_\mu, \overline K\phi_\mu\to 0$ as $\lambda\to\infty$. Hence
	\begin{equation}
		\abs{p_0}^2-\eps < \lim_{\lambda\to\infty} \norm{T\phi_\mu}_{L^2}^2 < \abs{p_0}^2+\eps,
	\end{equation}
	assuming the limit exists. Since $\eps$ was arbitrary, we conclude $\lim_{\lambda\to\infty} \norm{T\phi_\mu}_{L^2}^2=\abs{p_0}^2$.
\end{proof}

\subsection{Directly transmitted constituent of the parametrix}				\label{s:wave-ml}

For $T>0$, let $R_T$ be the solution operator for the wave equation~\eqref{e:wave-eqn} on $\RR^n$ with wave speed $c$.
As is well-known, $R_T$ is (away from glancing rays) the sum of graph FIOs associated with sequences of reflections and refractions. The first step is a microlocal diagonalization.

Let $B\in\Psi^1(\RR^n\setminus\Gamma)$ be a pseudodifferential square root of the elliptic spatial operator $-c^2\Delta$; choose a parametrix $B^{-1}\in\Psi^{-1}(\RR^n\setminus\Gamma)$. Away from $\Gamma$,
\begin{equation}
	\d_t^2-c^2\Delta		\eqml 	(\d_t-iB)(\d_t+iB).
\end{equation}
The factors $\d_t+iB$, $\d_t-iB$ are responsible for propagating singularities $(x,\xi)$ in the initial data forward and backward along bicharacteristics, respectively. If $u_\pm$ are solutions to $(\d_t\pm iB)u_\pm\eqml0$, then $u=u_++u_-$ solves $(\d_t^2-c^2\Delta)u\eqml0$. If $g_\pm=u_\pm(0,\cdot)$, then $u$ has Cauchy data $(h_{0},h_{1}) = (g_++g_-,\,iBg_+-iBg_-)$. Conversely, given $h=(h_0,h_1)$ and solving for $g_+$, $g_-$,
\begin{equation}
	\mat{g_-\\g_+} \eqml \frac12\mat{I  & \phantom{-} iB^{-1} \\ I & -iB^{-1}}\mat{h_{0}\\h_{1}}.
	\label{e:split-CD-from-standard}
\end{equation}
Let $\Lambda\colon(g_+,g_-)\mapsto(h_0,h_1)$. Then $R=\Lambda\smash{\smat{R^+\\&R^-}} \Lambda^{-1}$ for operators $R^+$ and $R^-$ which are order-0 FIOs away from glancing.

Given $y\in\Omega_r$, let $T=d(y,\bdy\Omega)$ and suppose $\gamma_y$ intersects $\Gamma$ exactly $k$ times.
Define $\dt^+(y)$ to be the principal symbol of the directly transmitted component $\RDT_k^+$ of $R^+$ at $(p(y),\nu)$, where $\nu$ is the inward-pointing normal covector at $p$. More precisely, in the notation of~\cite[Appendix A]{SC},
\begin{equation}
	\RDT_k^+ = \when{r_T \JCS, & k = 0,\\
				r_T\JBS\iota M\subT (\JBB\iota M\subT)^{k-1}\JCB, & k > 0.}
	\label{e:DT-op-def}
\end{equation}
It can be shown that
\begin{equation}
	\abs{\dt^+(y)} = \prod_{i=1}^k \frac{2\sqrt{\cot\alpha_i\cot\beta_i}}{\cot\alpha_i+\cot\beta_i},
	\label{e:dt-ps}
\end{equation}
where $\alpha_i$, $\beta_i$ are the angles between $\gamma'$ and the normal to $\Gamma$ at the $i\mith$ intersection of $\gamma$ with $\Gamma$.

\section{Interface recovery}      \label{s:interface-recovery}

We are now ready to apply the results of the previous sections and demonstrate how the discontinuities of $c$ can be located in boundary normal coordinates using outside measurements. The basic idea is to track the energy of a conormal wave packet as it travels into $\Omega$; each time it passes through a discontinuity in $c$ a known fraction of its energy is lost to reflection. As usual, a high-frequency limit is employed.


We begin with a result on recovery of the direct transmission's principal symbol, using wave packets.

\begin{theorem}
	\label{t:energy-ps-recovery}
	Let $y\in\Omega_r$, $p=p(y)$, $T=d(y,p)$, and let $\eps>0$ be sufficiently small. Then there exists a domain $\Theta\supset\Omega$ and a covector $(p^*\!,\nu^*)\in S^*\Theta$ such that
	\begin{align}
		\abs*{\dt^+(y)}^2 &= 	\lim_{\lambda\to\infty} \KE_{\Theta_{T+\eps}} R_{T+\eps} h_{0,\lambda},
		&
		h_{0,\lambda} &= \Lambda \mat{-icB^{-1}\rho_{\lambda,p^*\!,\nu^*}\phi_{\lambda,p^*\!,\nu^*}\\0}.
	\end{align}
\end{theorem}
The key interest in Theorem~\ref{t:energy-ps-recovery} is that $\KE_{\Theta_{T+\eps}}R_{T+\eps}h_{0,\lambda}$ is the kinetic energy of the almost direct transmission of wave packet $h_{0,\lambda}$. With scattering control, it can be obtained from measurements outside $\Omega$~\cite[Props.\ 2.7, 2.8]{SC}.

According to~\eqref{e:dt-ps}, $\dt^+(y)$ is smooth (in fact, constant) along each normal broken geodesic, except at discontinuities in $c$. This means scattering control can recover the discontinuities of $c$ in boundary normal coordinates as a direct consequence of Theorem~\ref{t:energy-ps-recovery}, and this recovery is completely constructive.

\begin{mainthm}
Assume $c$ is discontinuous on $\Gamma$, and let $y\in \Omega_r$, $T=d(y,\bdy\Omega)$. Then the locations of the singularities intersected by the normal broken geodesic segment $\gamma_y$ (in geodesic normal coordinates) are uniquely determined by the outside measurement operator $\mathcal F$, and given by
\begin{equation}
	\gamma_y^{-1}(\Gamma) = \singsupp\big(\abs{\dt^+\circ \gamma_y}\big).
\end{equation}
\label{c:interface-location}
\end{mainthm}

\begin{proof}[Proof of Theorem~\ref{t:energy-ps-recovery}]
We indicate just one method for choosing $\Theta$, noting that many others are possible. Namely, let $\Theta=\Omega_{-2\eps}$; that is, $\Theta$ is the $2\eps$-neighborhood of $\Omega$. Assume $\eps$ is sufficiently small that no two distinct geodesics normal to $\bdy\Omega\cap B_{4T}(p)$ intersect before reaching $\Theta$ (that is, no caustics form near $p$). Then $d^*(x,\Theta)=d^*(x,\Omega)+2\eps$ for any $x\in B_{2T}(p)\cap\Omega$.

We next choose the wave packet covector $(p^*,\nu^*)$. Define $\gamma$ as the maximal unit-speed geodesic with $\gamma(0)=p$ and $\gamma'(0)$ the inward normal to $\bdy\Omega$. Let $(p^*,\nu^*)=\gamma'^\flat(-\eps)$, and $\mu=\mu(\lambda)=(\lambda,p^*\!,\nu^*)$. For the rest of the proof, assume $\lambda$ is sufficiently large that $\supp\rho_\mu\subset\Theta\setminus\clsr\Omega$: the wave packet's cutoff lies inside the initial data region.

Now, we examine the energy distribution of the wavefields generated by corresponding wave packets at time $T$. In particular, we would like to show that the region $\Theta_{T+\epsilon}$, whose energy we probe with scattering control, contains only the directly transmitted component of the wavefield, in the high-frequency limit. If there were no glancing rays on any reflected branches, we could directly apply Proposition~\ref{p:principal-symbol-recovery} to conclude the proof. Instead, we follow a more careful argument.

To this end, we will decompose the energy of the wavefields generated by corresponding wave packets at time $T$. Since $y$ is a regular point, $\gamma$ intersects only finitely many interfaces, and each intersection is transversal. Let $\tilde t_1,\dotsc,\tilde t_m$ be the times of intersection. Let $\gamma_1,\dotsc,\gamma_m$ be the (unit-speed) reflected geodesics, parameterized so that $\gamma_i(\tilde t_i)=\gamma(\tilde t_i)$. Now choose slightly later times $t_1,\dotsc,t_m$ such that
\begin{equation}
	\tilde t_1 < t_1 < \tilde t_2 < t_2 < \dotsc < \tilde t_m < t_m < T.
\end{equation}
such that $\gamma_i$ intersects no interfaces in the time interval $(\tilde t_1,t_1]$. 

After each intersection, we capture the reflected energy with cutoffs $\alpha_1,\dotsc,\alpha_m$. Namely, let $\alpha_i$ to be a smooth bump function equal to 1 in a neighborhood of $\gamma_i(t_i)$ and supported away from $\Gamma\cup\{\gamma(t_i)\}$. Because $\gamma_i$ is not a minimal length path from $\gamma_i(t_i)$ to $\bdy\Theta$ we can choose $\supp\alpha_i$ small enough that $d(\supp\alpha_i,\bdy\Theta)<t_i+\eps-\delta$ for some $\delta>0$ (independent of $\eps$). Figure~\ref{f:theorem-c} illustrates the setup.

\begin{figure}
	\centering
	\includegraphics{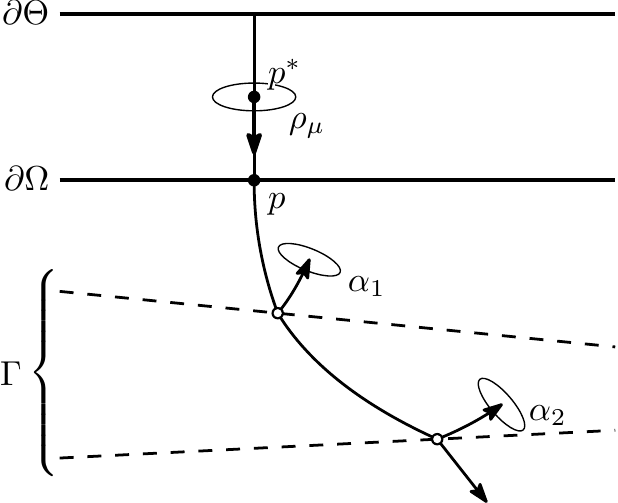}
	\caption{Cutoffs $\alpha_i$ used in the proof of Theorem~\ref{t:energy-ps-recovery}.}
	\label{f:theorem-c}
\end{figure}

Then we may divide $R_{T+\eps}$ into reflected and directly transmitted components as follows:
\begin{nalign}
R_{T+\eps} &= R_{t_2+\dotsc+t_m}\alpha_1 R_{t_1+\eps}\\
&\qquad + R_{t_3+\dotsc+t_m}\alpha_2 R_{t_2} (1-\alpha_1) R_{t_1+\eps}\\
&\qquad + R_{t_4+\dotsc+t_m}\alpha_3 R_{t_3} (1-\alpha_2) R_{t_2} (1-\alpha_1) R_{t_1+\eps}\\
&\qquad + \dotsb + \\
&\qquad + \alpha_m R_{t_m} (1-\alpha_{m-1}) R_{t_{m-1}} \dotsb (1-\alpha_1)R_{t_1+\eps} \\
&\qquad + (1-\alpha_m) R_{t_m} (1-\alpha_{m-1}) R_{t_{m-1}} \dotsb (1-\alpha_1)R_{t_1+\eps}.
\label{e:rt-decomposition}
\end{nalign}
Assuming now that $\eps$ is chosen smaller than $\delta$, finite speed of propagation ensures that the first $m$ (reflected) terms in~\eqref{e:rt-decomposition} vanish on $\Theta_{T+\eps}$, leaving only the final (transmitted) term.
\begin{equation}
	\RDT \eqml (1-\alpha_m) R_{t_m} (1-\alpha_{m-1}) R_{t_{m-1}} \dotsb (1-\alpha_1)R_{t_1+\eps},
\end{equation}
this equivalence modulo smoothing operators holding in a conic neighborhood of $(p^*,\nu^*)$.
A single graph FIO is required for applying Proposition~\ref{p:principal-symbol-recovery}, , so we define $\smash{\widetilde\RDT}^+=c^{-1}B\RDT^+cB^{-1}$, a graph FIO of order 0. Then the vanishing of the reflected terms in~\eqref{e:rt-decomposition} implies
\begin{equation}
	\KE_{\Theta_{T+\eps}} R_{T+\eps} \Lambda\mat{-icB^{-1}\rho_\mu\phi_\mu\\0}
	=
	\norm[\Big]{\smash{\widetilde\RDT}^+\rho_\mu\phi_\mu}_{L^2(\Theta_{T+\eps})}^2.
\end{equation}
By Proposition~\ref{p:principal-symbol-recovery},
\begin{equation}
	\lim_{\lambda\to\infty} \norm[\Big]{(\smash{\widetilde\RDT}^+\!+K)\rho_\mu\phi_\mu}_{L^2(\Theta_{T+\eps})}^2
	=
	\abs*{s(p^*\!,\nu^*)}^2,
\end{equation}
where $s$ is the principal symbol of $\smash{\widetilde\RDT}^+$, equal to that of $\RDT^+$. Since $R^+_\eps$ has a principal symbol of unity, $\abs{s(p^*\!,\nu^*)}^2=\abs{s(p,\nu)}^2=\abs{\dt^+(y)}^2$.
\end{proof}

\paragraph{Funding Acknowledgements:} P.~C.~and V.~K.~were supported by the Simons Foundation                 
under the MATH $+$ X program. M.~V.~dH.~was partially supported by the Simons Foundation                 
under the MATH $+$ X program, the National Science Foundation under                 
grant DMS-1559587, and by the members of the Geo-Mathematical Group at              
Rice University. G.~U.~is Walker Family Endowed Professor of Mathematics at the University of Washington, and is partially supported by NSF, a Si-Yuan Professorship at HKUST, and a FiDiPro Professorship at the Academy of Finland.

\bibliographystyle{siam}
\bibliography{ScatteringControlIP}

\end{document}